\documentclass[11pt,a4paper]{article}

\usepackage[sorting=none, sorting=nyt, maxbibnames=99, backend=biber]{biblatex}

\renewbibmacro{in:}{}
\usepackage[utf8]{inputenc}
\usepackage[T1]{fontenc}
\usepackage{amsmath,amssymb,amsthm}
\usepackage{geometry}
\usepackage{graphicx}
\usepackage{hyperref}
\usepackage{cleveref}
\usepackage{enumitem}
\usepackage{booktabs}
\usepackage{xcolor}
\usepackage{setspace}

\theoremstyle{definition}

\newtheorem{theorem}{Theorem}[section]
\newtheorem{lemma}[theorem]{Lemma}

\theoremstyle{remark}

\title{An Explicit Result for the Sum of Two Almost Primes}
\author{Adrian Dudek\thanks{School of Mathematics and Physics, The University of Queensland} \and Lachlan Dunn\footnotemark[1]}
\date{\today}

\begin{document}

\maketitle

\begin{abstract}
We show that every $N \geq 2$ can be written as the sum of positive integers $a$ and $b$ where $\Omega(ab) \leq 40$. The result is obtained through the direct application of an explicit lower bound Selberg sieve along with some computation and optimisation.
\end{abstract}

\section{Introduction}

There has been a recent expansion in using sieve methods to garner explicit results in the theory of numbers. In particular, using the explicit linear sieve developed by Bordignon, Johnston and Starichkova \cite{BJS}, the latter two authors went on to prove the following result (\cite{JohnstonStarichkova}).
\begin{theorem}
Every even integer $N \geq 4$ can be written as the sum of a prime and a positive integer $r$ such that $\Omega(r) \leq 395$. 
\end{theorem}
In the same paper, Johnston and Starichkova also prove that if one assumes the generalised Riemann hypothesis, then one can take $\Omega(r) \leq 31$ in the above. Of course, if one is contented by sufficiently large even $N$, then without any conditions we have the famous result of Chen \cite{ChenGoldbach} that one may take $\Omega(r)\leq2$.

In this paper, we consider what can be said for the result of Johnston and Starichkova if we approach the binary problem without fixing one of the two variables in $N = a + b$ to be prime. We also drop the condition that $N$ is even. This becomes a similar but less-constrained problem, and so our process is far simpler than Johnston and Starichkova, whose work had to include intricate arguments involving the possible impact of Siegel zeroes. 

The purpose of this paper is to prove the following theorem, which holds for all positive integers greater than one.

\begin{theorem} \label{maintheorem}
    Let $N \geq 2$. Then there exist positive integers $a$ and $b$ such that $N = a+b$ and $\Omega(ab) \leq 40$.
\end{theorem}

Although one may also approach this problem using the explicit linear sieve, the resulting calculations involve versions of the Bombieri--Vinogradov theorem that would ultimately provide a worse result than our approach. As such, we apply an explicit version of Selberg's $\Lambda^-\Lambda^2$ lower bound sieve from Friedlander and Iwaniec \cite{FI2010}. Using their sieve as a starting point, the approach is fairly direct and employs mostly standard estimates in analytic number theory along with some optimisation and computation. 

We start by setting up the sieve in the standard way (see, for example, Chapter 5 of \cite{FI2010}). Let $A$ be a finite sequence of positive integers, and let $|A|$ denote the number of terms of the sequence. Let $\mathcal{P}$ be the set of all prime numbers, and define
\begin{equation}\label{Pz}
    P(z):=\prod_{\substack{p<z\\p\in\mathcal{P}}}p.
\end{equation}
Let $S(A,z)$ denote the number of terms of the sequence $A$ that are not divisible by any prime $p \in \mathcal{P}$ such that $p < z$.
Furthermore, for any square-free integer $d$, let $|A_d|$ denote the number of terms of the sequence $A$ that are divisible by $d$, and let $g(d)$ be a multiplicative function with
\begin{equation*}
    0\leq g(p)<1\ \text{for all}\ p\in\mathcal{P}.
\end{equation*}
Define the remainder term
\begin{equation*}
    r_d(A):=|A_d|-|A|g(d).
\end{equation*}
We further consider a standard condition on the function $g$, namely that
\begin{equation}\label{sieveDimension}
    \prod_{w \leq p < z} (1-g(p))^{-1} \leq K \bigg( \frac{\log z}{\log w} \bigg)^{\kappa}
\end{equation}
for all $z > w \geq 2.$ Finally, we define
\begin{equation}
    V(z) := \prod_{p \mid P(z)} \left(1- g(p)\right).
\end{equation}
With the set up complete, we call on Theorem 7.7 from Friedlander and Iwaniec \cite{FI2010} which is an explicit form of Selberg's $\Lambda^-\Lambda^2$ lower bound sieve.

\begin{theorem} \label{LowerBoundSieveTheorem}
Let $g(d)$ be a density function which satisfies (\ref{sieveDimension}). Associated with it let $\Lambda = \Lambda^{-} \Lambda^{+} = \Lambda^{-} \Lambda^{2}$ be the lower-bound sieve of level $D$ described in \cite{FI2010}. Let $s:=\log D / \log z$, $k := \kappa + \log K$ and $X:=|A|$. Then we have, provided $s\geq 2k + 3$,
\begin{equation}\label{eq:7.121}
S(A, z) \geq XV(z) \left\{1 - \frac{s+3}{2e^k} \left(\frac{2ek}{s-3}\right)^{(s-3)/2}\right\} - 2R_4(A, D),
\end{equation}
where
\[
R_4(A, D) = \sum_{\substack{d\mid P(z) \\ d<D}} \tau_4(d)|r_d(A)|,
\]
and $\tau_4(d)$ denotes the 4-fold divisor function.
\end{theorem}

The proof of Theorem \ref{maintheorem} is a direct application of the above explicit sieve result along with some optimisation and computation. We first apply the sieve to find some $N_0$ such that Theorem \ref{maintheorem} is true for all $N \geq N_0$. We then employ computational arguments to fill in the remaining range. 

We begin in Section~\ref{sec:lems} by constructing explicit bounds on each of the terms used in Theorem~\ref{LowerBoundSieveTheorem}. We then prove the main result in Section~\ref{analyticSection}, and discuss possible improvements to our result in Section~\ref{improvements}.

\section{Preliminary Lemmas}\label{sec:lems}

In this section, we furnish the estimates that are required in the application of the sieve. Let $N$ be a positive integer and let $a_n=n(N-n)$ if $N$ is even and $a_n = n(N-n)/2$ if $N$ is odd. We define our sequence to be $A=(a_n)_{n=1}^{N-1}$. The value $g(p)$ will estimate the proportion of terms in $A$ that are divisible by $p$. It is clear that $g(2)=1/2$. Additionally, if $p$ (where $p \neq 2$) divides $N$, then $p\mid n$ implies that $p\mid N-n$, and hence $g(p)=1/p$. Otherwise, $p\mid n$ implies that $p\nmid N-n$ and $p\mid N-n$ implies that $p\nmid n$, so that $g(p)=2/p$. Hence, we have $$g(p) = \begin{cases}
    \frac{1}{2} & \text{if }p=2, \\
    \frac{2}{p} & \text{if $p\neq2$ and $p \nmid N$}, \\
    \frac{1}{p} & \text{if $p\neq2$ and $p \mid N$}.
\end{cases}$$ See for example Chapter 7 of Nathanson \cite{Nathanson} for a more extensive explanation.

Given that $\kappa = 2$ (as we will be sieving two residue classes for each appropriate prime $p$), we seek to find the least such value $K$ such that (\ref{sieveDimension}) holds uniformly for $z > w \geq 2$. We note that our result $K=3$ is best possible for the given sieving range (arising in Case 4 of the following proof).  Our result is as follows.

\begin{lemma}\label{lem:K}
    Let $g$ be the multiplicative arithmetic function given by
    $$g(p) = \begin{cases}
    \frac{1}{2} & \text{if }p=2, \\
    \frac{2}{p} & \text{if $p\neq2$ and $p \nmid N$}, \\
    \frac{1}{p} & \text{if $p\neq2$ and $p \mid N$}.
\end{cases}$$
 We have $$\prod_{w \leq p < z}(1-g(p))^{-1} \leq  3\left(\frac{\log z}{\log w} \right)^2$$ for $2 \leq w<z$.
\end{lemma}

\begin{proof}

We break the proof into four cases. The computer verification of Cases 2 and 4 was completed using the ``verifyK.jl'' file \cite{gitrepo}.

\singlespacing
\noindent \textbf{Case 1: $w \geq 286$}

From the Taylor series expansion of $\log(1-u)$ at $u=0$, we have that
\begin{align*}
    \log\left(\prod_{w \leq p\leq z}(1-g(p))^{-1}\right) &\leq \sum_{w \leq p\leq z}-\log\left(1-\frac{2}{p}\right)\\
    &\leq 2\sum_{w \leq p\leq z}\frac{1}{p}+\sum_{w\leq p\leq z}\sum_{k=2}^\infty\frac{1}{2}\left(\frac{2}{p}\right)^k \\
    & = 2\sum_{w \leq p\leq z}\frac{1}{p}+\sum_{w \leq p\leq z}\frac{2}{p(p-2)}.
\end{align*}
Using Theorem 5 of Rosser and Schoenfeld \cite{Rosser1962} we have that (for $z \geq 286)$
\begin{align*}
     2\sum_{w \leq p\leq z}\frac{1}{p}< 2\log\left(\frac{\log z}{\log w}\right) + \frac{1}{\log^2z}+\frac{1}{\log^2 w},
\end{align*}
so in turn we have the bound
\begin{align*}
    \log\left(\prod_{w \leq p\leq z}(1-g(p))^{-1}\right) &\leq 2\log\left(\frac{\log z}{\log w}\right) + \frac{1}{\log^2z}+\frac{1}{\log^2 w}+\sum_{w \leq p\leq z}\frac{2}{p(p-2)}\\
    &\leq 2\log\left(\frac{\log z}{\log w}\right) + 0.21,
    \end{align*}
where we have noted that 
$$\sum_{w\leq p \leq z} \frac{1}{p(p-2)} \leq \sum_{3\leq p} \frac{1}{p(p-2)} - \sum_{3 \leq p \leq 285} \frac{1}{p(p-2)} \leq \frac{1}{3}+\sum_{p \geq 3}\frac{1}{p^2} - 0.463 < 0.073.$$
Therefore, 
$$\prod_{w \leq p\leq z}(1-g(p))^{-1} \leq 1.24\left(\frac{\log z}{\log w}\right)^2,$$
and hence $K \leq 1.24$ for $w \geq 286$.

\singlespacing
\noindent \textbf{Case 2: $3 \leq w \leq 285,z \geq 286$}
    
We have
\begin{align*}
    \prod_{w \leq p\leq z}(1-g(p))^{-1} &\leq \left(\prod_{w \leq p\leq 285}\left(1-\frac{2}{p}\right)^{-1}\right)\left( \prod_{286 \leq p\leq z}(1-g(p))^{-1}\right)\\
    & \leq 1.24\left(\frac{\log w}{\log(286)}\right)^2\left( \prod_{w \leq p\leq 285}\Big(1-\frac{2}{p}\Big)^{-1} \right) \left(\frac{\log z}{\log w}\right)^2,
\end{align*}
where we have bounded the second product using Case 1. By computation of all values of $$1.24\left(\frac{\log w}{\log(286)}\right)^2\left( \prod_{w \leq p\leq 285}\left(1-\frac{2}{p}\right)^{-1} \right)$$ with $3\leq w\leq 285$ (and $w$ prime), we see that $K \leq 1.86$ for $3 \leq w \leq 285$ and $z \geq 286$.

\singlespacing
\noindent \textbf{Case 3: $w=2,z \geq 286$}

By applying Case 2, we have that
\begin{align*}
    &\prod_{2\leq p\leq z}(1-g(p))^{-1} = 2 \prod_{3 \leq p \leq z}(1-g(p))^{-1} \leq 2\cdot 1.86 \left(\frac{\log^2z}{\log^2(3)}\right)\leq 1.49\left(\frac{\log^2z}{\log^2(2)}\right),
\end{align*}
and hence $K \leq 1.49$ for $w=2,z \geq 286$.

\singlespacing
\noindent \textbf{Case 4: $z<286$}

From computer verification, we have that $K \leq 3$ for $z < 286$ (which occurs when $w=3$, $z=w+\epsilon$ for small $\epsilon>0$). This completes the proof.
\end{proof}

We also need an estimate for $V(z)$ at the range of $z$ for which we intend to bound the sieve analytically. In our case, this estimate is required for $z \geq 10^{10}$.

\begin{lemma}\label{lem:V}
    We have that
    $$V(z)  \geq \frac{0.2749}{\log^2(z)}$$
    for $z \geq 10^{10}$.
\end{lemma}
\begin{proof}
    To start, we note
\begin{align*}
    V(z) = \prod_{p<z}(1-g(p)) \geq \frac{1}{2}\prod_{2<p<z}\Big(1-\frac{2}{p}\Big).
\end{align*}
From the Taylor series expansion of $\log(1-u)$ at $u=0$, we have
\begin{align*}
    -\log\Big(\prod_{2<p<z}\Big(1-\frac{2}{p}\Big)\Big) = \sum_{2<p<z}-\log\Big(1-\frac{2}{p}\Big) = 2\sum_{2<p<z}\frac{1}{p}+ R(z),
\end{align*}
where $R(z) = \sum_{2<p<z}\sum_{k=2}^\infty\frac{1}{k}(\frac{2}{p})^k$. By Theorem 5 of Rosser and Schoenfeld \cite{Rosser1962}, we have that
\begin{align*}
    \sum_{2<p<z}\frac{1}{p} =\sum_{p<z}\frac{1}{p}- \frac{1}{2} < \log\log z -0.2366 \text{ for } z \geq 10^{10}.
\end{align*}
Additionally, we have 
\begin{align*}
    \frac{1}{2}R(z) \leq \sum_{2<p<z}\sum_{k=2}^\infty\frac{1}{4}\Big(\frac{2}{p}\Big)^k = \sum_{2<p<z} \frac{1}{p(p-2)} < \frac{1}{3}+\sum_{2<p<z}\frac{1}{p^2}<0.5357.
\end{align*}
Hence the result follows.
\end{proof}

The following lemma provides a bound on the remainder term for our sieve. Notably, we make the choice that $D = z^{18.4}$ for our application, that is, $s = 18.4$, and this can be seen in the statement of the lemma. 

\begin{lemma}\label{lem:R}
We have that
    $$R_4(A,z^{18.4}) < 0.591\, z^{18.4} \log^8 z$$
for all $z \geq 10^{10}.$
\end{lemma}
\begin{proof}
    Note that (see the proof of Theorem 7.2 in Nathanson \cite{Nathanson} for example)
$$|r_d(A)| \leq 2^{\omega(d)} = \tau(d)$$
and therefore
\[
R_4(A, D) \leq \sum_{\substack{d\mid P(z) \\ d<D}} \tau_4(d)\cdot \tau(d) = \sum_{\substack{d\mid P(z) \\ d<D}} \tau_8(d).
\]
An application of Rankin's trick gives us that
\begin{align*}
    R_4(A, D) & \leq \sum_{\substack{d\mid P(z) \\ d<D}} \tau_8(d) \\
    & < \sum_{\substack{d\mid P(z) \\ d<D}} \tau_8(d) \Big(\frac{D}{d}\Big)^{\delta} \\
    & = D^{\delta} \sum_{\substack{d\mid P(z) \\ d<D}} \frac{\tau_8(d)}{d^{\delta}} \\
    & < D^{\delta} \sum_{d\mid P(z)} \frac{\tau_8(d)}{d^{\delta}} \\
    & = D^{\delta} \prod_{p < z} \Big(1 + \frac{8}{p^{\delta}} \Big).
\end{align*}
Employing the estimate $1+x < e^x$ for $x > 1$ gives us that
$$R_4(A, D) < D^{\delta} \exp \bigg( 8 \sum_{p < z} \frac{1}{p^{\delta}}\bigg).$$
We write $\delta = 1 - \eta$ with $\eta > 0$ to be chosen later. Applying the inequality $e^x < 1 + xe^{x}$ for $x > 0$ gives us
\begin{align*}
    R_4(A, D) &< D^{1- \eta} \exp\bigg(8 \sum_{p<z} \frac{1}{p} \cdot e^{\eta \log p} \bigg) \\
    &< D^{1- \eta} \exp\bigg(8 \sum_{p<z} \frac{1}{p} (1+\eta \log p \, z^{\eta}) \bigg) \\
    &= D^{1- \eta} \exp\bigg(8 \sum_{p<z} \frac{1}{p}\bigg) \exp\bigg( 8 \eta z^{\eta} \sum_{p < z} \frac{\log p}{p} \bigg). 
\end{align*}
By Theorem 5 of Rosser and Schoenfeld \cite{Rosser1962}, for $z \geq 10^{10}$ we have that
\begin{align*}
    &\sum_{p < z} \frac{1}{p} < \log \log z + 0.2634, \\
    &\sum_{p<z} \frac{\log p}{p} < \log z.
\end{align*}
Direct substitution yields the bound
$$R_4(A, D) < D^{1- \eta} \exp(8 \log \log z + 2.1072 + 8 \eta z^{\eta} \log z).$$
Now, we let $\eta = \alpha/\log z$ where $\alpha > 0$ will be chosen later so as to optimise. We get that
\begin{equation}
    R_4(A,D) < D \log^8 z \,\exp\bigg(2.1072 + 8 \alpha e^{\alpha} - \alpha \frac{\log D}{\log z} \bigg).
\end{equation}
By fixing $\frac{\log D}{\log z}=18.4$, and by setting $\alpha=0.457$, the result follows.
\end{proof}

\section{Proof of Theorem \ref{maintheorem}} \label{analyticSection}
We will first use the following lemma to deal with small values of $N$.

\begin{lemma}\label{lem:primegaps}
    Let $m \in \mathcal{P} \cup \{1\}$ and $M \in \mathbb{Z}^+$. If $m \cdot 1476\leq2^{M-2}$, then for every $N \leq m \cdot 4 \cdot 10^{18}$ there exists a decomposition $N=a+b$ such that $\Omega(ab) \leq M$.
\end{lemma}
\begin{proof}
    Let $P := \{p \leq 4 \cdot 10^{18} : p \in \mathcal{P} \}$, and $mP := \{mp : p \in P \}$. Let $p_i$ denote the $i$th prime number and $\pi(n-):=\lim_{x \to n^-}\pi(x)$. By Herzog, Pardi and Silva \cite{eSilva2014}, we have $p_{i+1}-p_i \leq 1476$ for $p_i \in P$. For $N \leq m \cdot 4 \cdot 10^{18}$, we note that $N-mp_{\pi(\frac{N}{m}-)} \leq mp_{\pi(\frac{N}{m}-)+1} - mp_{\pi(\frac{N}{m}-)}\leq m \cdot 1476$, and so $\Omega(N-mp_{\pi(\frac{N}{m}-)}) \leq \log_2(m \cdot 1476)\leq M-2$. Hence we have $N = mp_{\pi(\frac{N}{m}-)} + (N-mp_{\pi(\frac{N}{m}-)})$ and $\Omega(mp_{\pi(\frac{N}{m}-)}) + \Omega(N-mp_{\pi(\frac{N}{m}-)}) \leq M$.
\end{proof}

We will use the following lemma to deal with large values of $N$.

\begin{lemma}\label{lem:main}
    Let $N \in \mathbb{N}$ with $N \geq 2$ and $M \in \mathbb{R}_{>1}$, and suppose that $S(A,N^{1/M})>0$ for $N \geq N_0$. Then there exists $a,b \in \mathbb{N}$ such that $N=a+b$ with $\Omega(ab) < 2M$ if $N$ is even and $\Omega(ab) < 2M+1$ if $N$ is odd.
\end{lemma}
\begin{proof}
    Let $r$ be an element of $A$ that remains after sieving (and note that $r$ can be written as either $n(N-n)$ or $n(N-n)/2$). First, we note that $r\leq n(N-n)\leq \frac{N}{2}\left(N-\frac{N}{2}\right)= \frac{N^2}{4}$. Now, as all of $r$'s primes factors are greater than or equal to $N^{1/M}$, if $\Omega(r) \geq 2M$ we would have $r \geq N^\frac{2M}{M}=N^2$, which contradicts the first point. Thus, $\Omega(r) < 2M$. 
    
    Now we define $a:=n$ and $b:=N-n$. If $N$ is even, then $2M > \Omega(r)=\Omega(ab)$. If $N$ is odd, then $2M > \Omega(r)=\Omega(ab)-1$. This completes the proof.
\end{proof}

We now prove the main result.

\begin{theorem}\label{thm:last}
    Let $N \in \mathbb{N}$. There exists $a,b \in \mathbb{N}$ such that $N=a+b$ with $\Omega(ab) \leq 39$ if $N$ is even and $\Omega(ab) \leq 40$ if $N$ is odd.
\end{theorem}
\begin{proof}

We will split into three cases which correspond to small, medium and large values of $N$. For the first case, we will prove the theorem directly. For the second and third case, we will show that $S(A,z)>0$ for $z=N^{1/20}$ and apply Lemma~\ref{lem:main}.

\singlespacing

    \noindent \textbf{Case 1: $N \leq 4 \cdot 10^{18}$}
    
    Apply Lemma~\ref{lem:primegaps} with $m=1$.

\singlespacing 

    \noindent \textbf{Case 2: $4 \cdot 10^{18} <N \leq 10^{200}$}

    In this section, we aim to show that $S(A,z)>0$ for $z$ in the required range. To compute this value exactly for each $z$ would be computationally infeasible. As such, we employ various techniques to speed up the computation. Firstly, we employ the bound
    $$R_4(A,D)<D^\delta\prod_{p<z}\left(1+\frac{8}{p^\delta}\right), $$
    which was derived mid-way through the proof of Lemma~\ref{lem:R}. Secondly, we will construct a function $S(p,q)$ (where $p,q$ are primes and $p<q$) such that $S(A,z)>S(p,q)$ for $z\in (p,q]$. In this manner, we need only verify that $S(p,q)>0$ to show that $S(A,z)>0$ for all $z\in (p,q]$. We now look to define the function $S(p,q)$.

    Let $p$ and $q$ be prime numbers with $p<q$, and let $p_i$ denote the $i$th prime number. For any $p<z:=N^{1/r}\leq q$, we have from Theorem \ref{LowerBoundSieveTheorem} that
    \begin{equation*}
        S(A,z) \geq (p^{r}-1)V(q)\left\{1 - \frac{s+3}{2e^k} \left(\frac{2ek}{s-3}\right)^{(s-3)/2}\right\} - 2q^{s\delta}\prod_{p'<q}\left(1+\frac{8}{(p')^\delta}\right)=:S(p,q).
    \end{equation*}
    To verify Case 2, we will construct a sequence of primes $(q_i)_{i=1}^m \subset \mathcal{P}$ such that $2=q_1<q_2<\cdots< q_m=10^{10}+147$ and such that, for each $i=1,\ldots,m-1$, $S(q_i,q_{i+1})>0$ for a suitable choice of $s$ and $\delta$. For our computations, we will fix $r=20$, $s=18.4$ and $\delta \in \{0.2, 0.8, 0.9\}$. We will define our sequence of primes such that $q_{i+1}-q_i$ increases as $i$ increases. In particular, we define
    $$q_i = \begin{cases}
        p_{\pi(q_{i-1})+1} & 1 \leq \pi(q_{i-1}) < 100 \\
        p_{\pi(q_{i-1})+10^{\ell-1}} & 10^{\ell}\leq \pi(q_{i-1})< 10^{\ell+1} \text{ for } \ell \geq 2.
    \end{cases}$$
    Using the ``verifyS.jl'' file \cite{gitrepo}, it was verified (within 5 minutes) that, for all $i=1,\ldots,m-1$, there exists a $\delta \in \{0.2,0.8,0.9\}$ such that $S(q_i,q_{i+1})>0$. Hence Case 2 of the theorem follows.

\singlespacing
    \noindent \textbf{Case 3: $N>10^{200}$}
    
    We apply Theorem \ref{LowerBoundSieveTheorem} with parameters $s=18.4, k=2+\log(3)$ and $ z=N^{1/20}$. By applying Lemmas~\ref{lem:K}, \ref{lem:V} and \ref{lem:R}, we have that
    \begin{align*}
        S(A,z) \geq \frac{4(N-1)}{\log^2N}-\frac{1.182}{20^8}N^{0.92}\log^8N>0
    \end{align*}
    for all $N > 10^{200}$.
\end{proof}

\section{Future Work}\label{improvements}
To begin, we note that one can use our methods to get an $N_0$ such that Theorem~\ref{maintheorem} holds for the case where $\Omega(ab) \leq 37$ and $N\geq N_0$. However, the subsequent computer verification of Case 2 of Theorem~\ref{thm:last} would become computationally infeasible. Additionally, the value 37 can not be lowered by a direct application of Theorem~\ref{LowerBoundSieveTheorem}. In our proof, we show that $S(A,N^{1/20})>0$ for $N \geq N_0$. In particular, we require that the main term of the sieve asymptotically dominates the error term, so that such an $N_0$ exists. If we instead choose $z=N^{1/M}$ for some $M \in \mathbb{R}_{>1}$, then the main term is $O(\frac{N}{\log^2N})$ and the error term is $O(N^{s/M}\log^8(N))$. Thus we require $M>s$. However, for the main sieve term to be positive, we require that 
$$1 - \frac{s+3}{2e^k} \left(\frac{2ek}{s-3}\right)^{(s-3)/2}>0$$
which implies that $18.3<s<M$. Therefore, by applying Lemma~\ref{lem:main}, the lowest obtainable value of $M$ (using our methods) would be 37.

To resolve this issue, one could potentially furnish a better value for $K$ by reworking the detailed proof of Theorem \ref{LowerBoundSieveTheorem} so that a different range could be taken for $w$ and $z$. In particular, one could then take $K$ close to 1 and deal with the smaller primes by hand. This would then allow an $N_0$ to be furnished for $\Omega(ab) \leq M$ where $M<40$. Additionally, it is highly likely that employing a weighted sieve would improve things still further; see \cite{DudekJohnston} for the explicit application of Kuhn's weighted sieve as well as \cite{JohnstonSorensonThomasWebster} for a similar treatment involving so-called Richert weights. Additionally, Franze \cite{Franze} has published refinements of Selberg's sieve that, if made explicit, could improve the result further. 

\section{Acknowledgements}

The authors wish to thank Daniel Johnston for sharing this problem with us, which was suggested to him by a reviewer of his PhD thesis. We thank DJ further for pointing us at Theorem \ref{LowerBoundSieveTheorem} in Opera de Cribro \cite{FI2010} as the ready-to-fire tool needed for the problem. We also thank Peter Campbell for correcting some mistakes in an earlier version of the result. Thanks also must go to the referee for their helpful comments.

\printbibliography

\end{document}